\documentclass[10pt]{amsart}
\usepackage{amssymb,amsbsy,amsmath,amsfonts,times, amscd}
\usepackage{latexsym,euscript,exscale}

\usepackage{helvet}

\title{Coarse equivalence and topological couplings of locally compact groups}

\author{Uri Bader}
\address{Faculty of Mathematics and Computer Science\\
Weizmann Institute of Science\\
234 Herzl Street\\
Rehovot 7610001\\
Israel
}
\email{bader@weizmann.ac.il}
\urladdr{http://www.weizmann.ac.il/math/uribader/}
\author {Christian Rosendal}
\address{Department of Mathematics, Statistics, and Computer Science (M/C 249)\\
University of Illinois at Chicago\\
851 S. Morgan St.\\
Chicago, IL 60607-7045\\
USA}
\email{rosendal.math@gmail.com}
\urladdr{http://homepages.math.uic.edu/$~$rosendal}

\date {}
\newcommand{\norm}[1]{\lVert#1\rVert}

\newcommand{\NORM}[1]{\Big\lVert#1\Big\rVert}

\newcommand {\R}{\mathbb R}

\newcommand{\om}{\omega}
\newcommand{\eps}{\epsilon}

\newcommand{\equi}{\Leftrightarrow}

\newcommand{\ov}{\overline}
\newcommand{\inv}{^{-1}}

\newtheorem{thm}{Theorem}%[section]
\newtheorem{cor}[thm]{Corollary}

\newtheorem{claim}[thm] {Claim}

\theoremstyle{definition}

\usepackage[sc]{mathpazo}
\begin{document}
%\subjclass[2000]{Primary: 20E08, Secondary: 03E15}

\thanks{The second author was partially supported by  the NSF (DMS 1464974)}

\maketitle

\begin{abstract}
In \cite{gromov}, M. Gromov showed that any two finitely generated groups $\Gamma$ and $\Lambda$ are quasi-isometric if and only if they admit a topological coupling, i.e., a commuting pair of proper continuous cocompact actions $\Gamma\curvearrowright X \curvearrowleft \Lambda$ on a locally compact Hausdorff space. This result is extended here to all (compactly generated) locally compact second countable groups.
\end{abstract}

\

\

%\section{Theorem}
In his seminal work on geometric group theory  \cite{gromov}, M. Gromov formulated a topological criterion for quasi-isometry of finitely generated groups (see $0.2.C_2'$  in  \cite{gromov}). His idea was to replace the geometric objects, namely, the finitely generated groups, by a purely topological framework, namely, a locally compact Hausdorff space, which has no intrinsic large scale geometric structure.

As it is, Gromov's proof easily adapts to characterise coarse equivalence of arbitrary countable discrete groups, but thus far the case of locally compact groups has not been adressed and indeed Gromov's construction is insufficient to deal with these. The present paper presents a solution to this problem by establishing the following theorem.

\begin{thm}\label{main}
Two locally compact second countable groups  are coarsely equivalent if and only if they admit a topological coupling. 
\end{thm}

As coarse equivalence of locally compact, compactly generated groups is just quasi-isometry, we have the following corollary.
\begin{cor}
Two compactly generated, locally compact second countable groups are quasi-isometric if and only if they admit a topological coupling.
\end{cor}

Let us recall that a {\em topological coupling} of two locally compact groups $G$ and $H$ is a pair $G\curvearrowright X \curvearrowleft  H$ of commuting, proper and cocompact continuous actions on a locally compact Hausdorff space $X$. Here the actions are {\em proper} if, for every compact subset $K\subseteq X$, the sets
$$
\{g\in G\mid gK\cap K\neq \emptyset\}, \quad \{h\in H\mid Kh\cap K\neq \emptyset\}
$$
are both compact. Also, the actions are cocompact if $X=G\cdot K=K\cdot H$ for some compact subset $K\subseteq X$.

A {\em coarse equivalence} between two metric spaces $X$ and $Y$ is a map $\phi\colon X\to Y$ for which $\phi[X]$ is {\em cobounded} in $Y$, i.e., $\sup_{y\in Y}d(y, \phi[X])<\infty$, and so that, for all sequences $x_n, z_n$ in $X$, we have
$$
\lim_nd(x_n,z_n)=\infty \;\;\equi \;\; \lim_nd(\phi(x_n), \phi(z_n))=\infty.
$$
Alternatively, the latter condition may be expressed by saying that there are functions $\kappa,\omega\colon \R_+\to \R_+$ with $\lim_{t\to \infty}\kappa(t)=\infty$ so that
$$
\kappa\big(d(x,z)\big)\leqslant d(\phi(x), \phi(z))\leqslant \omega\big(d(x,z)\big)
$$
for all $x,z\in X$.

On the other hand, a {\em quasi-isometry} is a coarse equivalence in which the bounding functions $\kappa$ and $\omega$ may be taken to be affine.

We note that every locally compact second countable group $G$ admits a compatible left-invariant {\em proper} metric $d$, i.e., whose balls are compact. Moreover, any two such metrics turn out to be coarsely equivalent and hence define a unique coarse geometry on $G$. If, in addition, $G$ is compactly generated, say $G=\langle K\rangle$ for some symmetric compact set $K$, there is a compatible left-invariant proper metric $d$ that is quasi-isometric to the left-invariant word metric $\rho_K$ induced by $K$. And as before, any two compact symmetric generating sets induce quasi-isometric word metrics and hence define a unique quasi-isometry type of $G$. Finally, any coarse equivalence between compactly generated, locally compact second countable  groups will in fact be a quasi-isometry.

\

We now proceed to the proof of Theorem \ref{main}.
\begin{proof}
Suppose $G$ and $H$ are locally compact second countable groups and let $d_G$ and $d_H$ be proper left-invariant compatible metrics on $G$ and $H$ respectively. The existence of such metrics is guaranteed by \cite{struble}. Let also $\mu$ denote left-invariant Haar measure on $G$ scaled so that the closed unit ball $B=\ov {B_{d_G}(1_G, 1)}$ has measure $1$. 

Assume that $\phi\colon H\to G$ is a coarse equivalence and let $\kappa_\phi, \om_\phi$ be respectively the compression and the expansion moduli of $\phi$, 
$$
\kappa_\phi(t)=\inf_{d_H(h_1,h_2)\geqslant t}d_G(\phi h_1, \phi h_2),
$$
$$
\om_\phi(t)=\sup_{d_H(h_1,h_2)\leqslant t}d_G(\phi h_1, \phi h_2)
$$
Choose also $s>0$ so that $\kappa_\phi(s)\geqslant 3$.

Let $Y\subseteq H$ be a maximal $s$-discrete subset, i.e., so that $d_H(y, y')\geqslant s$ and hence also $d_G\big(\phi(y), \phi(y')\big)\geqslant 3$ for all $y\neq y'$ in $Y$. By maximality, $Y$ is $s$-dense in $H$, i.e., for every $h\in H$ there is some $y\in Y$ with $d_H(h,y)<s$. For every $y\in Y$, define $\theta_y\colon H\to [0,s+1]$  by $\theta_y(h)=\max\{0,s+1- d_H(h,y)\}$. Note that $\theta_y$  is  $1$-Lipschitz and $\theta_y\geqslant 1$ on the ball of radius $s$ centred at $y$, while $\theta_y=0$ outside the ball of radius $s+1$. By properness of the metric $d_H$, it follows that
$$
\Theta(h)=\sum_{y\in Y}\theta_y(h)
$$
is a bounded Lipschitz function with $\Theta\geqslant 1$. 

Let now $Z=\phi[Y]$ and, for $z=\phi(y)$, define $\alpha_z=\frac{\theta_y}{\Theta}$.
Then the family $\{\alpha_{z}\}_{z\in Z}$ is a partition of unity of $H$ consisting of $N$-Lipschitz functions $\alpha_z\colon H\to [0,1]$ for some $N>0$, so that 
$$
{\rm supp}(\alpha_{\phi(y)})\subseteq B_{d_H}(y, s+1),
$$
for each $y\in Y$. Thus, for all $h\in H$, we have
$$
Z_h=\{z\in Z\mid \alpha_z(h)>0\}   \subseteq B_{d_G}\big(\phi(h), \om_\phi(s+1)\big).
$$
Note that $Z_h$ 
is $3$-discrete. We let $M$ be the maximal size of a $3$-discrete set of diameter at most $2\om_\phi(s+1)$ in $G$.

Let $\lambda\colon G\curvearrowright L^1(G,\mu)$ denote the left-regular representation and define a  function $\psi\colon H\to L^1(G,\mu)$ by
$$
\psi_h=\sum_{z\in Z}\alpha_z(h)\cdot \chi_{zB}=\sum_{z\in Z_h}\alpha_z(h)\cdot \chi_{zB} \in L^1(G,\mu).
$$
Thus  each $\psi_h$ is the convex combination of at most $M$  disjointly supported shifts $\lambda(z)\chi_B=\chi_{zB}$ of the characteristic function $\chi_B$. Therefore, for $h,f\in H$, 
\[\begin{split}
\norm{\psi_h-\psi_f}_{L^1}
&=\NORM{\sum_{z\in Z_h\cup Z_f}(\alpha_z(h)-\alpha_z(f))\cdot \chi_{zB}}_{L^1}\\
&\leqslant \sum_{z\in Z_h\cup Z_f}|\alpha_z(h)-\alpha_z(f)|\cdot \norm {\chi_{zB}}_{L^1}\\
&\leqslant \sum_{z\in Z_h\cup Z_f}N\cdot d_H(h,f)\\
&\leqslant 2MN\cdot d_H(h,f).
\end{split}\]
I.e., $\psi$ is $2MN$-Lipschitz.
Also, $\norm{\psi_h}_{L^1}=1$, while 
$$
{\rm supp}(\psi_h) \subseteq \ov{B_{d_G}(\phi(h), \om_\phi(s+1)+1)}.
$$

Set 
\[\begin{split}
X=
\{\xi\in L^1(G,\mu)\mid& \;\;  \xi=\sum_{i=1}^m\alpha_i\chi_{g_iB} \text{ where  $\{g_1,\ldots , g_m\}$ is a $3$-discrete subset of }\\
&\text{$G$ with diameter $\leqslant 2\om_\phi(s+1)$  and $\alpha_i\geqslant 0$ with $\sum_{i=1}^m \alpha_i=1$}\}.
\end{split}\]
Observe that $X$ is invariant under the left-regular presentation $\lambda\colon G\curvearrowright L^1(G,\mu)$.

\begin{claim}
$X$ is locally compact in the norm topology on $L^1(G,\mu)$. In fact, 
$$
[K, \eps]=\{\xi \in X\mid \langle \xi\mid \chi_K\rangle \geqslant \eps\}
$$
is norm compact for every compact set  $K\subseteq G$ and $\eps>0$. Conversely, every compact subset of $X$ is contained in some $[K,\eps]$.
\end{claim}

\begin{proof}
As $[K,\eps]$ is closed in $X$, it suffices to show that  every sequence  $\xi_n$  in $[K,\eps]$ has a convergent subsequence. So let $\xi_n$ be given. By passing to a subsequence, there is some $m\leqslant M$ so that each $\xi_n$ can be written as a convex combination
$$
\xi_n=\sum_{i=1}^m\alpha_{i,n}\chi_{g_{i,n}B}
$$
of some $3$-discrete subset  $\{g_{1,n},\ldots , g_{m,n}\}\subseteq G$ with diameter at most $2\om_\phi(s+1)$.
Note then that, as $\langle \xi_n\mid \chi_K\rangle >0$, we have $g_{i,n}B\cap K\neq\emptyset$ for some $i\leqslant m$ and thus that $d_G(g_{j,n},K)\leqslant 2\om_\phi(s+1)+1$ for all $j$. Therefore, by passing to a further subsequence, we may assume that $g_i=\lim_ng_{i,n}$ and $\alpha_i=\lim_n\alpha_{i,n}$ exist for all $i\leqslant m$. It follows that $\{g_1,\ldots, g_m\}$ is a $3$-discrete subset of $G$ with radius at most $2\om_\phi(s+1)$, that $\alpha_i\geqslant 0$ with $\sum_i\alpha_i=1$ and that 
$$
\xi=\sum_{i=1}^m\alpha_i\chi_{g_iB}\in X
$$
is the norm limit of the $\xi_n$. 

Suppose now instead that $C\subseteq X$ is compact. Then $C$ is covered by open sets of the form  
$$
(K, \eps)=\{\xi \in X\mid \langle {\xi}\mid \chi_K\rangle > \eps\}
$$
for $K$ compact and $\eps>0$ and hence may be covered by finitely many of these, $C\subseteq \bigcup_{i=1}^p(K_i,\eps_i)$. It thus follows that $C\subseteq [\bigcup_{i=1}^pK_i,\min_{i=1}^p\eps_i]$.
\end{proof}

Consider now the space of maps $X^H$ and note that $\psi\in X^H$.
Endow $X^H$ with the product topology and commuting left and right actions $G\curvearrowright X^H \curvearrowleft H$ by homeomorphisms given by
$$
(g\cdot \xi)_h=\lambda(g)\xi_h, \quad (\xi\cdot h)_f=\xi_{hf}
$$
for $g\in G$, $h,f\in H$ and $\xi \in X^H$.
We set $\Omega=\ov{G\cdot \psi\cdot H}$.
Note that $\Omega$ is $G\times H$-invariant. We will see that $\Omega$ is a topological coupling of $G$ and $H$.

Note first that, for $g\in G$ and $h, f,f_1, f_2\in H$,
\[\begin{split}
\norm{(g\psi h)_{f_1}-(g\psi h)_{f_2}}_{L^1}
&=\norm{\lambda(g)\big(    \psi _{hf_1} - \psi_{hf_2}  \big)   }_{L^1} \leqslant 2MN\cdot d_H(f_1,f_2)
\end{split}\]
and conclude that for all $\xi \in \Omega$ and $f_1, f_2\in H$, 
\[\begin{split}
\norm{\xi_{f_1}-\xi_{f_2}}_{L^1}\leqslant 2MN\cdot d_H(f_1,f_2).
\end{split}\]
We also have 
$$
{\rm supp}\big((g\psi h)_f\big)
={\rm supp}\big(\lambda(g)\psi_{hf}\big)
=g\cdot {\rm supp}\big(\psi_{hf}\big) 
\subseteq \ov{B_{d_G}(g\phi(hf), \om_\phi(s+1)+1)}.
$$
Therefore, 
\[\begin{split}
\kappa_\phi\big(d_H(f_1,f_2)\big)-2\om_\phi(s+1)-2
&\leqslant d_G\Big( {\rm supp}\big((g\psi h)_{f_1}\big), {\rm supp}\big((g\psi h)_{f_2}\big)\Big)\\
&\leqslant \om_\phi\big(d_H(f_1,f_2)\big)+2\om_\phi(s+1)+2.
\end{split}\]
By consequence, we have that for all $\xi \in \Omega$ and $f_1, f_2\in H$, 
\[\begin{split}
\kappa_\phi\big(d_H(f_1,f_2)\big)-2\om_\phi(s+1)-2
&\leqslant d_G\Big( {\rm ess\; supp}\big(\xi_{f_1}\big), {\rm ess\; supp}\big(\xi_{f_2}\big)\Big)\\
&\leqslant \om_\phi\big(d_H(f_1,f_2)\big)+2\om_\phi(s+1)+2.
\end{split}\]

\begin{claim}
The space $\Omega=\ov{G\cdot \psi\cdot H}$ is locally compact.
\end{claim}

\begin{proof}
Indeed, given $\xi\in \Omega$, let $K= {\rm ess\; supp}\big(\xi_1\big)$ and consider the neighbourhood 
$$
\{\eta\in \Omega \mid \eta_1\in  [K,1/2]\}
$$
of $\xi$.  Then, if $\zeta\in \{\eta\in \Omega \mid \eta_1\in  [K,1/2]\}$, we have $K\cap  {\rm ess\; supp}\big(\zeta_1\big)\neq \emptyset$ and thus
$$
{\rm ess\; supp}\big(\zeta_f\big)\subseteq (K)_{ \om_\phi(d_H(f,1))+6\om_\phi(s+1)+6 }.
$$
By consequence
$$
\{\zeta\in \Omega \mid \zeta_1\in  [K,1/2]\}\subseteq \prod_{f\in H}[(K)_{ \om_\phi(d_H(f,1))+6\om_\phi(s+1)+6 }, 1/2],
$$
where the latter product is compact.
\end{proof}

\begin{claim}
The action $\Omega \curvearrowleft H$ is continuous and proper.
\end{claim}

\begin{proof}
To show continuity at $\xi\in \Omega$ and $h\in H$, we must show that, for all  $f\in H$ and $\eps>0$, there are neighbourhoods $V$ and $W$ of $\xi$ and $h$ respectively so that
$\norm{(\xi\cdot h)_f-(\zeta\cdot k)_f}_{L^1}<\eps$ whenever $\zeta\in V$ and $k\in W$. 

So assume $f\in H$ and $\eps>0$ are given and let $\zeta\in V$ if $\norm{\xi_{hf}-\zeta_{hf}}_{L^1}<\frac \eps2$, while $k\in W$ if $d_H(hf, kf)<\frac \eps{4MN}$.
Then, for all $\zeta\in V$ and $k\in W$, we have 
\[\begin{split}
\norm{(\xi\cdot h)_f-(\zeta\cdot k)_f}_{L^1}
&\leqslant \norm{\xi_{hf}-\zeta_{hf}}_{L^1}+\norm{\zeta_{hf}-\zeta_{kf}}_{L^1} \\
&<\frac \eps2+2MN\cdot d_H(hf, kf)\\
&<\eps
\end{split}\]
as required.

To see that the action is proper, observe that for any compact subset of $\Omega$ the projection on the coordinate $1\in H$ is also compact, hence the subset is contained in a set of the form
$$
\{\xi\in \Omega\mid \xi_1\in [K,\eps]\},
$$
for some compact $K\subseteq G$ and $\eps>0$. 

So suppose that $\xi,\zeta\in \Omega$ satisfy $\xi_1, \zeta_1\in [K,\eps]$, while $h\in H$ satisfies
$$
\kappa_\phi\big(d_H(h,1)\big)> {\rm diam}_{d_G}(K)+2\om_\phi(s+1)+2.
$$
Then ${\rm ess\; supp}\big(\xi_{1}\big)$ and ${\rm ess\; supp}\big(\zeta_{1}\big)$ must both intersect $K$. Therefore, 
\[\begin{split}
{\rm diam}_{d_G}(K)
&<\kappa_\phi\big(d_H(h,1)\big)-2\om_\phi(s+1)-2\\
&\leqslant d_G\Big( {\rm ess\; supp}\big(\zeta_{h}\big), {\rm ess\; supp}\big(\zeta_{1}\big)\Big)
\end{split}\]
and so ${\rm ess\; supp}\big(\zeta_{h}\big)$ must be disjoint from $K$, whence ${\rm ess\; supp}\big(\zeta_{h}\big)\neq {\rm ess\; supp}\big(\xi_{1}\big)$. In particular, $(\zeta\cdot h)_1=\zeta_h\neq \xi_1$.

By consequence,
$$
\{\xi\in \Omega\mid \xi_1\in [K,\eps]\} \cap \{\xi\in \Omega\mid \xi_1\in [K,\eps]\}\cdot h=\emptyset
$$
for all $h\in H$ with $\kappa_\phi\big(d_H(h,1)\big)> {\rm diam}_{d_G}(K)+2\om_\phi(s+1)+2$, witnessing properness of the action.
\end{proof}

\begin{claim}
The action $\Omega \curvearrowleft H$ is cocompact.
\end{claim}

\begin{proof}
Fix $R>\sup_{g\in G}d_G(g, \phi[H])$
and let $K\subseteq G$ be the closed ball of radius $R+\om_\phi(s+1)+1$ centred at $1_G$. We will show that
$$
\Omega= \{\xi\in \Omega\mid \xi_1\in [K,1/2]\}\cdot H.
$$

To see this, fix $\zeta\in \Omega$ and set $C={\rm ess\; supp}\big(\zeta_{1}\big)$. Now, suppose that  ${\rm supp}\big((g\psi h)_1\big)$ intersects $C$ for some $g\in G$ and $h\in H$, whence $d_G(C,g\phi(h))\leqslant \om_\phi(s+1)+1$. Pick   $f\in H$ so that $d_G(g\phi(hf), 1_G)=d_G(\phi(hf), g\inv )<R$ and observe that then 
\[\begin{split}
\kappa_\phi\big(d_H(f,1)\big)
&\leqslant d_G( g\phi(h),g\phi(hf))\\
&\leqslant d_G(g\phi(h), C)+ {\rm diam}_{d_G}(C)+d_G(C, 1_G)+ d_G(1_G, g\phi(hf))\\
&\leqslant \om_\phi(s+1)+1+ {\rm diam}_{d_G}(C)+d_G(C, 1_G)+ R.
\end{split}\]
Letting 
$$
r=\sup\big({t}\mid \kappa_\phi(t)\leqslant  \om_\phi(s+1)+1+ {\rm diam}_{d_G}(C)+d_G(C, 1_G)+ R\big),
$$ 
we find that for all  $g\in G$ and $h\in H$ for which ${\rm supp}\big((g\psi h)_1\big)$ intersects $C$, there is some $f\in B_{d_H}(1_H, r)$ with $d_G(g\phi(hf), 1_G)<R$, whence also 
$$
{\rm supp}\big((g\psi h)_f\big)\subseteq K.
$$ 
and  $\langle (g\psi h)_f\mid \chi_K\rangle =1$.
It follows that $\zeta$ is a limit of points $g\psi h\in \Omega$ for which there are $f\in B_{d_H}(1_H, r)$ 
with $\langle (g\psi h)_f\mid \chi_K\rangle =1$.

Assume for a contradiction that $(\zeta\cdot f)_1=\zeta_f\notin [K,1/2]$ for all $f\in H$ and let $f_1,\ldots, f_n$ be $\frac 1{5MN}$-dense in $B_{d_H}(1_H, r)$.  Choose $g\psi h\in \Omega$ close enough to $\zeta$ so that $(g\psi h)_{f_i}\notin [K,1/2]$ for all $i\leqslant n$, while, on the other hand, $\langle (g\psi h)_f\mid \chi_K\rangle =1$ for some $f\in B_{d_H}(1_H, r)$.
Pick then $i$ with $d_H(f_i, f)\leqslant \frac 1{5MN}$, whereby
\[\begin{split}
1/2
&\leqslant \big|\langle (g\psi h)_{f_i}\mid \chi_K\rangle- \langle (g\psi h)_f\mid \chi_K\rangle\big|\\
&\leqslant \norm{ (g\psi h)_{f_i}- (g\psi h)_{f}}_{L^1}\\
&\leqslant 2MN\cdot d_H(f_i, f)\\
&< 1/2,
\end{split}\]
which is absurd. 

Thus, $\zeta\in \{\xi\in \Omega\mid \xi_1\in [K,1/2]\}\cdot f$ for some $f\in H$. Since $\zeta\in \Omega$ was arbitrary and $\{\xi\in \Omega\mid \xi_1\in [K,1/2]\}$ relatively compact in $\Omega$, this proves cocompactness of the action.
\end{proof}

\begin{claim}
The action $G\curvearrowright \Omega $ is continuous, proper and cocompact.
\end{claim}

\begin{proof}
Continuity is trivial since already the action of $G$ on $X$ and hence on $X^H$ is continuous. For properness, it suffices to see that, for every $\eps>0$ and compact set $1\in K\subseteq G$, the set of $g\in G$ for which 
$$
\{\xi\in \Omega\mid \xi_1\in [K,\eps]\}\cap g\cdot \{\xi\in \Omega\mid \xi_1\in [K,\eps]\}\neq\emptyset
$$
is relatively compact in $G$. But, if $\xi_1\in [K,\eps]$ and $\lambda(g)\xi_1=(g\cdot \xi)_1\in [K,\eps]$, then ${\rm ess\; supp}\big(\xi_{1}\big)$ intersects both $K$ and $g\inv K$, whereby
$$
d_G(K,g\inv K)\leqslant {\rm diam}_{d_G}\big({\rm ess\; supp}\big(\xi_{1}\big)\big)\leqslant 2\om_\phi(s+1)+2
$$
and thus $d_G(g, 1)\leqslant 2\om_\phi(s+1)+2+2\,{\rm diam}_{d_G}(K)$.

Finally, for cocompactness, let $K$ be the ball of radius $4\om_\phi(s+1)+4$ centred at $1_G$. Then every subset of $G$ of diameter at most $2\om_\phi(s+1)+2$ can be translated via an element of $G$ into $K$. But this means that, if $\xi\in \Omega$, then 
$$
{\rm ess\; supp}\big((g\cdot \xi)_{1}\big)=g\cdot {\rm ess\; supp}\big(\xi_{1}\big)\subseteq K
$$
for some $g\in G$ and thus $g\cdot \xi\in \{\zeta\in \Omega\mid \zeta_1\in [K,1]\}$. In other words,
$$
\Omega=G\cdot  \{\zeta\in \Omega\mid \zeta_1\in [K,1]\},
$$
showing cocompactness.
\end{proof}
As both of the actions are proper, cocompact and continuous, we have a topological coupling. 
\end{proof}

As stated, Theorem \ref{main} only makes sense for actions of locally compact groups. Nevertheless, with appropriate adjustments, one may formulate and prove a generalisation for a larger class of Polish groups, namely those with bounded geometry. This is the subject of a forthcoming paper \cite{bounded geom}.

\end{document}